\documentclass[12pt,reqno,a4paper]{article}

\usepackage[outer=2.5cm,top=2.5cm,bottom=2.5cm,inner=2.0cm]{geometry}

\usepackage{xfrac}
\usepackage{hyperref} 
\usepackage{amsmath}
\usepackage{amsthm}
\usepackage{amssymb}
\usepackage{amsfonts}
\usepackage{dsfont}
\usepackage{stmaryrd}
\usepackage{commath}
\usepackage{mathrsfs}
\usepackage{enumitem}
\usepackage{graphicx}
\usepackage{tikz}
\usepackage{algorithm}
\usepackage{algorithmic}
\usepackage{pgfplots}
\usepackage{subfig}
\usepackage[numbers]{natbib}

\usepackage{todonotes}
\usepackage{placeins}

\font\msbm=msbm10

\numberwithin{equation}{section}

\theoremstyle{plain}
\newtheorem{theorem}{Theorem}[section]
\newtheorem{lemma}[theorem]{Lemma}
\newtheorem{corollary}[theorem]{Corollary}
\newtheorem{proposition}[theorem]{Proposition}

\theoremstyle{definition}

\newtheorem{remark}[theorem]{Remark}
\def\mathbb#1{\hbox{\msbm{#1}}}

\newcommand{\field}[1]{\ensuremath{\mathds{#1}}}

\newcommand{\re}{\field R}\newcommand{\N}{\field N}
\newcommand{\Ept}{\field E}\newcommand{\Prob}{\field P}

\newcommand{\C}{\field{C}}
\newcommand{\tor}{\field T} 

\newcommand{\Z}{\field Z}

\newcommand{\R}{\field R}
\newcommand{\T}{\field T}

\newcommand{\Id}{\ensuremath{\mathrm{Id}}}

\newcommand{\bk}{\mathbf{k}}
\newcommand{\bj}{\mathbf{j}}
\newcommand{\bx}{\mathbf{x}}
\newcommand{\bw}{\mathbf{w}}
\newcommand{\by}{\mathbf{y}}

\newcommand{\bL}{\mathbf{L}}

\newcommand{\bX}{\mathbf{X}}

\newcommand{\bU}{\mathbf{U}}
\newcommand{\bV}{\mathbf{V}}

\newcommand{\be}{\begin{equation}}
\newcommand{\ee}{\end{equation}}
\newcommand{\beq}{\begin{eqnarray}}
\newcommand{\beqq}{\begin{eqnarray*}}
\newcommand{\eeq}{\end{eqnarray}}
\newcommand{\eeqq}{\end{eqnarray*}}

\newcommand{\trace}[1]{\operatorname{tr}(#1)}

\newcommand{\diag}{\mathrm{diag}}

\begin{document}

\title{A new upper bound for sampling numbers}
\author{Nicolas Nagel, Martin Sch\"afer, Tino Ullrich\footnote{Corresponding author: tino.ullrich@mathematik.tu-chemnitz.de} \\\\
TU Chemnitz, Faculty of Mathematics, 09107 Chemnitz, Germany}

\maketitle

\begin{abstract}
\sloppy We provide a new upper bound for sampling numbers $(g_n)_{n\in \N}$ associated to the compact embedding of a separable reproducing kernel Hilbert space into the space of square integrable functions. There are universal constants $C,c>0$ (which are specified in the paper) such that
$$
	g^2_n \leq \frac{C\log(n)}{n}\sum\limits_{k\geq \lfloor cn \rfloor} \sigma_k^2\quad,\quad n\geq 2\,,
$$
where $(\sigma_k)_{k\in \N}$ is the sequence of singular numbers (approximation numbers) of the Hilbert-Schmidt embedding $\Id:H(K) \to L_2(D,\varrho_D)$. The algorithm which realizes the bound is a least squares algorithm based on a specific set of sampling nodes. These are constructed out of a random draw in combination with a down-sampling procedure coming from the celebrated proof of Weaver's conjecture, which was shown to be equivalent to the Kadison-Singer problem. Our result is non-constructive since we only show the existence of a linear sampling operator realizing the above bound. The general result can for instance be applied to the well-known situation of $H^s_{\text{mix}}(\tor^d)$ in $L_2(\tor^d)$ with $s>1/2$. We obtain the asymptotic bound
$$
	g_n \leq C_{s,d}n^{-s}\log(n)^{(d-1)s+1/2}\,,
$$
which improves on very recent results by shortening the gap between upper and lower bound to $\sqrt{\log(n)}$. The result implies that for dimensions $d>2$ any sparse grid sampling recovery method does not perform asymptotically optimal.


\medskip

\small%
\end{abstract}

\section{Introduction} 
In this paper we study a well-known problem on the optimal recovery of multivariate functions from $n$ function samples. The problem turned out to be rather difficult in several relevant situations. Since we want to recover the function $f$ from $n$ function samples $(f(\bx^1),...,f(\bx^n))$ the problem boils down to the question of how to choose these sampling nodes $\bX = (\bx^1,...,\bx^n)$ and corresponding recovery algorithms. The minimal worst-case error for an optimal choice is reflected by the $n$-th sampling number defined by
\begin{equation}\label{sn}
	g_n(\Id_{K,\varrho_D}) := \inf\limits_{\bx^1,...,\bx^n \in D} \; \inf\limits_{\varphi:\C^n \to L_2} \; \sup\limits_{\|f\|_{H(K)}\leq 1}\|f-\varphi(f(\bx^1),...,f(\bx^n))\|_{L_2(D,\varrho_D)}\,.
\end{equation}
The functions are modeled as elements from a separable reproducing kernel Hilbert space $H(K)$ of functions on a set $D \subset \R^d$ with finite trace kernel $K(\cdot,\cdot)$, i.e.,
\begin{equation}\label{trace}
	\trace{K}:=\int_D K(\bx,\bx)d\varrho_D(\bx)<\infty\,.
\end{equation}
The recovery problem (in the above framework) has been first addressed by G. Wasilkowski and H. Wo{\'z}niakowski in \cite{WaWo01}. The corresponding problem for certain particular cases (e.g. classes of functions with mixed smoothness properties, see \cite[Sect.\ 5]{DuTeUl19}) has been studied much earlier. Our main result is the existence of two universal constants $C,c>0$ (specified in Remark \ref{constants}) such that the relation
\begin{equation}\label{f00}
	g_n^2 \leq C \frac{\log(n)}{n}\sum\limits_{k\geq \lfloor cn \rfloor} \sigma_k^2\quad, \quad n\geq 2,
\end{equation}
holds true between the sampling numbers $(g_n)_{n\in \N}$ and the square summable singular numbers $(\sigma_k)_{k\in \N}$ of the compact embedding
$$
	\Id_{K,\varrho_D}:H(K) \to L_2(D,\varrho_D)\,.
$$
We emphasize that, in general, the square-summability of the singular numbers $(\sigma_k)_{k\in \N}$ is not implied by the compactness of the embedding $\Id_{K,\varrho_D}$. This is one reason why we need the additional assumption of a finite trace kernel \eqref{trace} (or a Hilbert-Schmidt embedding). In addition, as it has been observed by A. Hinrichs, E. Novak and J. Vyb{\'i}ral \cite{HiNoVy08}, the non-existing trace may cause the sampling numbers to have a  worse (or even no) polynomial decay than the corresponding polynomially decaying singular numbers $(\sigma_k)_k$. Hence, an inequality (like \eqref{f00}) which passes on the polynomial decay of the singular numbers to the sampling numbers is in general impossible without the condition of a finite trace \eqref{trace}. In our main example, the recovery of multivariate functions with dominating mixed smoothness (see Section  \ref{discussion}), this condition is equivalent to $s>1/2$, where $s$ denotes the mixed smoothness parameter. For further historical and technical comments (e.g.\ non-separable RKHS) we refer to Remark \ref{MoeUl}.

The algorithm which realizes the bound \eqref{f00} in the sense of \eqref{sn} is a (linear) least squares algorithm based on a specific set of sampling nodes. These are constructed out of a random draw in combination with a down-sampling procedure coming from the proof of Weaver's conjecture \cite{NiOlUl16}, see Section \ref{weaver}. In its original form the result in \cite{NiOlUl16} is not applicable for our purpose. That is why we have to slightly generalize it, see Theorem \ref{thm_weaver} below. Note that the result in \eqref{f00} is non-constructive. We do not have a deterministic construction for a suitable set of nodes. However, we have control of the failure probability which can be made arbitrarily small. In addition, the subspace, where the least squares algorithm is taking place is precisely given and determined by the first $m$ singular vectors.

The problem discussed in the present paper is tightly related to the problem of the Marcinkiewicz discretization of $L_2$-norms for functions from finite-dimensional spaces (e.g. trigonometric polynomials). In fact, constructing well-conditioned matrices for the least squares approximation is an equivalent issue. Let us emphasize that V.N. Temlyakov  (and coauthors) already used the S. Nitzan, A. Olevskii and A. Ulanovskii construction \cite{NiOlUl16} for the Marcinkiewicz discretization problem in the context of multivariate (hyperbolic cross) polynomials, see \cite{Te18, Te18_1} and the very recent paper \cite{LiTe20}.


Compared to the result by D. Krieg and M. Ullrich \cite{KrUl19} the relation \eqref{f00} is stronger. In fact, the difference is mostly in the $\log$-exponent as the example below shows. The general relation \eqref{f00} yields a significant improvement in the situation of mixed Sobolev embeddings in $L_2$, see Section \ref{discussion}. Applied for instance to the situation of $H^s_{\text{mix}}(\tor^d)$ in $L_2(\tor^d)$ with $s>1/2$ (this condition is equivalent to the finite trace condition \eqref{trace}) the result in \eqref{f00} yields
\begin{equation}\label{mix}
	g_n \lesssim_d n^{-s}\log(n)^{(d-1)s+1/2}\,,
\end{equation}
whereas the result in \cite{KrUl19} (see also \cite{KUV19,Ul20,MoUl20}) implies
$$
	g_n \lesssim_d n^{-s}\log(n)^{(d-1)s+s}\,.
$$
The $\log$-gap grows with $s>1/2$.
Our new result achieves rates that are only worse by $\sqrt{\log(n)}$ in comparison to the benchmark rates given by the singular numbers. Note that in $d\geq 3$ and any $s>1/2$ the bound \eqref{mix} yields a better performance than any sparse grid technique is able to provide, see \cite{Tem93}, \cite{BuGr04}, \cite{SiUl07}, \cite{DuUl15}, \cite{Du11}, \cite{By18} and \cite[Sect.\ 5]{DuTeUl19}. In addition, combining the above result with recent preasymptotic estimates for the $(\sigma_j)_j$, see \cite{KSU2}, \cite{KSU3}, \cite{Ku19}, \cite{Kr18}, we are able to obtain reasonable bounds for $g_n$ also in the case of small $n$. See Section \ref{discussion} for further comments and references in this direction.

D. Krieg and M. Ullrich \cite{KrUl19} used a sophisticated random sampling strategy which allowed for establishing a new connection between sampling numbers and singular values. Let us emphasize that this can be considered as a major progress in this field. In addition, the result in this paper partly relies on this random sampling strategy according to a distribution built upon  spectral properties of the embedding. The advantage of the pure random strategy in connection with a $\log(n)$-oversampling is the fact that the failure probability decays polynomially in $n$ which has been recently shown by M. Ullrich \cite{Ul20} and, independently, by M. Moeller together with the third named author \cite{MoUl20}. In other words, although this approach incorporates a probabilistic ingredient, the failure probability is controlled and the algorithm may be implemented. Note, that there are some obvious parallels to the field of compressed sensing, where also the measurement matrix is drawn at random and satisfies RIP with high probability.
\paragraph{Notation.} As usual $\N$ denotes the natural numbers, $\N_0:=\N\cup\{0\}$, $\Z$ denotes the integers, $\R$ the real numbers and \(\R_+\) the non-negative real numbers and $\C$ the complex numbers. For a natural number $m$ we set $[m] := \{1, ..., m\}$. We will also use $\dot{\cup}$ to emphasize, that a union is disjoint. If not indicated otherwise  $\log(\cdot)$ denotes the natural logarithm of its argument. $\C^n$ denotes the complex $n$-space, whereas $\C^{m\times n}$ denotes the set of all $m\times n$-matrices $\bL$ with complex entries. Vectors and matrices are usually typesetted boldface with $\bx,\by\in \C^n$. The matrix $\bL^{\ast}$ denotes the adjoint matrix. The spectral norm of matrices $\bL$ is denoted by $\|\bL\|$ or $\|\bL\|_{2\to 2}$.  For a complex (column) vector $\by\in \C^n$ (or $\ell_2$) we will often use the tensor notation for the matrix
$$
	\by \otimes \by := \by\cdot \by^\ast = \by \cdot \overline{\by}^\top 	
	\in\C^{n \times n}\;\textnormal{(or $\C^{\N\times\N}$)}\,.
$$
For $0<p\leq \infty$ and $\bx\in \C^n$ we denote $\|\bx\|_p := (\sum_{i=1}^n
|x_i|^p)^{1/p}$ with the usual modification in the case $p=\infty$ or $\bx$ being an infinite sequence. As usual we will denote with $\Ept X$ the expectation of a random variable $X$ on a probability space $(\Omega, \mathcal{A}, \Prob)$. Given a measurable subset $D\subset \R^d$ and a measure $\varrho$ we denote with $L_2(D,\varrho)$ the space of all square integrable complex-valued functions (equivalence classes) on $D$ with $\int_D |f(\bx)|^2\, d\varrho(\bx)<\infty$. We will often use $\Omega = D^n$ as probability space with the product measure $\Prob = d\varrho^n$ if $\varrho$ is a probability measure itself.

\section{Weaver's theorem}
\label{weaver}
In this section we prove a modified version of Weaver's $\text{KS}_2$-theorem, also known as Weaver's $\text{KS}_2$-conjecture, from~\cite{Wea2004} which was shown to be equivalent to the famous Kadison-Singer conjecture~\cite{KaSi1959} dating back as far as 1959. For a long time, these statements were mere conjectures and many people even believed them to be false. Since the celebrated proof given by A.~Marcus, D.~Spielman, and N.~Srivastava~\cite{MaSpSr15} in 2015, however, they
have turned into actual theorems and thus into rather strong tools for various applications,
and it is in fact the Weaver $\text{KS}_2$-conjecture that is at the heart of our argument in this article.
We need it in a slightly modified form, however, formulated in Theorem~\ref{thm_weaver} below. The starting point for its proof is the following reformulation of the classical Weaver statement which already occurred in~\cite{NiOlUl16}. We will formulate it with slightly improved constants, see \cite{Na20}.

\begin{theorem}[\cite{NiOlUl16}] \label{Thm:WeaverExplicite}
	Let $0 < \varepsilon$ and $\mathbf{u}_1, ..., \mathbf{u}_n \in \C^m$ with $\|\mathbf{u}_i\|_2^2 \leq \varepsilon$ for all $i=1, ..., n$ and
	\begin{align}\label{cond:tightfr}
	\sum_{i=1}^n |\langle \mathbf{w}, \mathbf{u}_i\rangle|^2 = \|\mathbf{w}\|_2^2
	\end{align}
	for all $\mathbf{w} \in \C^m$. Then there is a partition $S_1 \dot{\cup} S_2 = [n]$ with
	$$
	\sum_{i \in S_j} |\langle \mathbf{w}, \mathbf{u}_i\rangle|^2 \leq \frac{(1+\sqrt{2\varepsilon})^2}{2} \|\mathbf{w}\|_2^2
	$$
	for each $j=1, 2$ and all $\mathbf{w} \in \C^m$. Especially, we have
	$$
	\frac{1-(2+\sqrt{2})\sqrt{\varepsilon}}{2} \|\mathbf{w}\|_2^2 \leq \sum_{i \in S_j} |\langle \mathbf{w}, \mathbf{u}_i\rangle|^2 \leq \frac{1+(2+\sqrt{2})\sqrt{\varepsilon}}{2} \|\mathbf{w}\|_2^2
	$$
	for each $j=1, 2$ and all $\mathbf{w} \in \C^m$.
\end{theorem}

Note, that the above statement is trivial for $\varepsilon \geq (2+\sqrt{2})^{-2}$, since in this case the lower bound is $\leq 0$ and the upper bound is $\geq 1$. Relaxing condition~\eqref{cond:tightfr}, one obtains
an analogous statement for non-tight frames.

\begin{corollary}[\cite{NiOlUl16}] \label{Cor:WeaverInductive}
	Let  $0 < \varepsilon$ and $\mathbf{u}_1, ..., \mathbf{u}_n \in \C^m$ with $\|\mathbf{u}_i\|_2^2 \leq \varepsilon$ for all $i=1, ..., n$ and
	$$
	\alpha \|\mathbf{w}\|_2^2 \leq \sum_{i=1}^n |\langle \mathbf{w}, \mathbf{u}_i\rangle|^2 \leq \beta \|\mathbf{w}\|_2^2
	$$
	for all $\mathbf{w} \in \C^m$, where $\beta \geq \alpha > 0$ are some fixed constants. Then there is a partition $S_1 \dot{\cup} S_2 = [n]$, such that
	$$
	\frac{1-(2+\sqrt{2})\sqrt{\varepsilon/\alpha}}{2} \cdot \alpha \|\mathbf{w}\|_2^2 \leq \sum_{i \in S_j} |\langle \mathbf{w}, \mathbf{u}_i\rangle|^2 \leq \frac{1+(2+\sqrt{2})\sqrt{\varepsilon/\alpha}}{2} \cdot \beta \|\mathbf{w}\|_2^2
	$$
	for each $j=1, 2$ and all $\mathbf{w} \in \C^m$.
\end{corollary}

Again, the above statement is trivial for $\varepsilon/\alpha \geq (2+\sqrt{2})^{-2}$. Now we are ready to formulate and prove the theorem which is convenient for our later purpose. The proof technique of this theorem is analogous to the one used for the proof of Lemma 2 in \cite{NiOlUl16}. After the preprint was finished, V.N. Temlyakov pointed out to us that their proof of Lemma 2.2 in their recent paper \cite{LiTe20}, which is stated in a weaker form, also contains a version of the theorem below with unspecified constants.

\begin{theorem} \label{thm_weaver}
	Let $k_1, k_2, k_3 > 0$ and $\mathbf{u}_1, ..., \mathbf{u}_n \in \C^m$ with $\|\mathbf{u}_i\|_2^2 \leq k_1 \frac{m}{n}$ for all $i=1, ..., n$ and
	$$
	k_2 \|\mathbf{w}\|_2^2 \leq \sum_{i=1}^n |\langle \mathbf{w}, \mathbf{u}_i\rangle|^2 \leq k_3 \|\mathbf{w}\|_2^2
	$$
	for all $\mathbf{w} \in \C^m$. Then there is a $J \subseteq [n]$ of size $\# J \leq c_1 m$ with
	$$
	c_2 \cdot \frac{m}{n} \|\mathbf{w}\|_2^2 \leq \sum_{i \in J} |\langle \mathbf{w}, \mathbf{u}_i\rangle|^2 \leq c_3 \cdot \frac{m}{n} \|\mathbf{w}\|_2^2
	$$
	for all $\mathbf{w} \in \C^m$, where $c_1, c_2, c_3$ only depend on $k_1, k_2, k_3$. More precisely, we can choose
	\[
    c_1 = 1642 \frac{k_1}{k_2} \,,\quad
	c_2 = (2+\sqrt{2})^2k_1 \,,\quad
	c_3 = 1642 \frac{k_1k_3}{k_2}
	\]
    in case $\frac{n}{m}\geq 47 \frac{k_1}{k_2}$. In the regime $1\le\frac{n}{m}< 47 \frac{k_1}{k_2}$ one may put $c_1 = 47k_1/k_2$, $c_2 = k_2$, $c_3=47k_1 k_3/k_2$.
\end{theorem}
\begin{proof}
	To ease the notation a bit, let us set $\zeta := 2+\sqrt{2}$. Put
	$\delta := k_1 \frac{m}{n}$, $\alpha_0 := k_2$, $\beta_0 := k_3$ and define recursively
	$$
	\alpha_{\ell+1} := \frac{1-\zeta\sqrt{\delta/\alpha_\ell}}{2} \cdot \alpha_\ell, \quad \beta_{\ell+1} := \frac{1+\zeta\sqrt{\delta/\alpha_\ell}}{2} \cdot \beta_\ell
	$$
	for $\ell \in \N_0$. Assume for the moment that $\delta < (2\zeta)^{-2}k_2$. We want to show that there is a constant $\gamma>0$, not depending on $\delta$ and an $L \in \N$, such that $\alpha_\ell \geq (2\zeta)^2\delta$ for all $\ell \leq L$ as well as $\zeta^2 \delta \leq \alpha_{L+1} < (2\zeta)^2\delta$ and $\beta_{L+1} < \gamma\alpha_{L+1} < (2\zeta)^2\gamma\delta$. \\
	Notice that
	$$
	\frac{1-\sqrt{\delta/\alpha_\ell}}{2}
	$$
	is strictly increasing in $\alpha_\ell > 0$. For $\alpha_\ell \geq (2\zeta)^2\delta$ we thus have
	$$
	\frac{1}{4} = \frac{1-\zeta\sqrt{\frac{\delta}{(2\zeta)^2\delta}}}{2} \leq \frac{1-\zeta\sqrt{\delta/\alpha_\ell}}{2} < \lim\limits_{x \longrightarrow \infty} \frac{1-\zeta\sqrt{\delta/x}}{2} = \frac{1}{2},
	$$
	and therefore
	$$
	\frac{1}{4} \alpha_\ell \leq \frac{1-\zeta\sqrt{\delta/\alpha_\ell}}{2} \cdot \alpha_\ell = \alpha_{\ell+1} < \frac{1}{2} \alpha_\ell.
	$$
	Set $L := \max\left\{\ell\in\N_0: \alpha_\ell \geq (2\zeta)^2\delta\right\}$ so that $\alpha_\ell \geq (2\zeta)^2\delta$ for all $\ell \leq L$. Notice that, since $\alpha_{\ell+1} < \alpha_\ell/2$ as long as $\alpha_\ell \geq (2\zeta)^2\delta$, we have $L < \infty$. Since $\alpha_0 = k_2 > (2\zeta)^2\delta$, we also have $L \geq 0$. \\
	The definition of $L$ directly yields $\alpha_{L+1} < (2\zeta)^2\delta$, but by the above also
	$$
	\alpha_{L+1} \geq \frac{1}{4} \alpha_L \geq \zeta^2 \delta.
	$$
	It remains to find a $\gamma > 0$ as described above. To do so, first observe by the definition of the $\alpha_\ell$ and $\beta_\ell$ that
	\begin{align*}
	\frac{\beta_{L+1}}{\alpha_{L+1}} & = \frac{\beta_L}{\alpha_L} \cdot \frac{1+\zeta\sqrt{\delta/\alpha_L}}{1-\zeta\sqrt{\delta/\alpha_L}} = \frac{\beta_{L-1}}{\alpha_{L-1}} \cdot \frac{1+\zeta\sqrt{\delta/\alpha_{L-1}}}{1-\zeta\sqrt{\delta/\alpha_{L-1}}} \cdot \frac{1+\zeta\sqrt{\delta/\alpha_L}}{1-\zeta\sqrt{\delta/\alpha_L}} = ... \\
	& = \frac{\beta_0}{\alpha_0} \cdot \prod_{\ell=0}^L \frac{1+\zeta\sqrt{\delta/\alpha_\ell}}{1-\zeta\sqrt{\delta/\alpha_\ell}}.
	\end{align*}
	We have $\alpha_L \geq (2\zeta)^2\delta$ so that $\zeta\sqrt{\delta/\alpha_L} \leq 2^{-1}$ and using
	$$
	\alpha_{\ell+1} < \frac{1}{2} \alpha_\ell \quad \Rightarrow \quad \zeta\sqrt{\delta/\alpha_\ell} < 2^{-1/2} \cdot \zeta \sqrt{\delta/\alpha_{\ell+1}}
	$$
	inductively we get $\zeta \sqrt{\delta/\alpha_{L-\ell}} \leq 2^{-1-\ell/2}$ for $\ell = 0, 1, ..., L$. Thus
	$$
	\frac{\beta_{L+1}}{\alpha_{L+1}} \leq \frac{k_3}{k_2} \cdot \prod_{\ell=0}^L \frac{1+\zeta\sqrt{\delta/\alpha_\ell}}{1-\zeta\sqrt{\delta/\alpha_\ell}} < \frac{k_3}{k_2} \cdot \prod_{\ell=0}^\infty \frac{1+2^{-1-\ell/2}}{1-2^{-1-\ell/2}} =: \gamma < 35.21 \cdot \frac{k_3}{k_2},
	$$
	which yields the final claim.
	
	With this at hand, consider the situation of the theorem. Clearly, we have $n\ge m$ due to the lower frame bound $k_2>0$. We now distinguish two cases.
    Firstly, if $1\leq \frac{n}{m} < 47\frac{k_1}{k_2}$ the assertion follows directly for $J = [n]$ and the choice $c_1 = \frac{n}{m}$, $c_2 = k_2 \frac{n}{m}$, and $c_3 = k_3 \frac{n}{m}$. Incorporating the bounds for $n/m$ gives the choice in the statement of the theorem. \\
	In the second case, when $\frac{n}{m}\ge 47\frac{k_1}{k_2}$, let $\alpha_\ell$, $\beta_\ell$ be as above and note that $\delta = k_1 \frac{m}{n} \leq k_2/47 < (2\zeta)^{-2}k_2$.
    The vectors $\mathbf{u}_i$ fulfill the assumptions of Corollary \ref{Cor:WeaverInductive} for $\alpha = \alpha_0 = k_2$ and $\beta = \beta_0 = k_3$, so that there is a set $J_1 \subseteq [n]$ with
	$$
	\alpha_1 \|\mathbf{w}\|_2^2 \leq \sum_{i \in J_1} |\langle \mathbf{w}, \mathbf{u}_i\rangle|^2 \leq \beta_1 \|\mathbf{w}\|_2^2
	$$
	for all $\mathbf{w} \in \C^m$. By choosing the smaller of the two partition classes $S_1$ or $S_2$, we may assume $\# J_1 \leq \frac{1}{2} n$. We can now apply Corollary \ref{Cor:WeaverInductive} again, where we restrict ourselves to the indices in $J_1$. We thus get a $J_2 \subseteq J_1$ with
	$$
	\alpha_2 \|\mathbf{w}\|_2^2 \leq \sum_{i \in J_2} |\langle \mathbf{w}, \mathbf{u}_i\rangle|^2 \leq \beta_2 \|\mathbf{w}\|_2^2
	$$
	for all $\mathbf{w} \in \C^m$. Again, by choosing the smaller partition class, we may assume $\# J_2 \leq \frac{1}{2} \# J_1 \leq \frac{1}{4} n$. After $L+1$ applications of Corollary \ref{Cor:WeaverInductive}, we get
	$$
	\alpha_{L+1} \|\mathbf{w}\|_2^2 \leq \sum_{i \in J_{L+1}} |\langle \mathbf{w}, \mathbf{u}_i\rangle|^2 \leq \beta_{L+1} \|\mathbf{w}\|_2^2
	$$
	for all $\mathbf{w} \in \C^m$, where $J_{L+1} \subseteq [n]$ with $\#J_{L+1} \leq 2^{-(L+1)} n$. By what was proven in the first part of this proof, we therefore get
	$$
	\zeta^2\delta \|\mathbf{w}\|_2^2 \leq \sum_{i \in J_{L+1}} |\langle \mathbf{w}, \mathbf{u}_i\rangle|^2 \leq (2\zeta)^2 \gamma \delta \|\mathbf{w}\|_2^2
	$$
	for all $\mathbf{w} \in \C^m$. We thus get the assertion for $J = J_{L+1}$, $c_2 = \zeta^2 k_1$ and
    $c_3 = (2\zeta)^2 \gamma k_1\le 35.21\cdot(2\zeta)^2 \frac{k_1k_3}{k_2}$. As for $c_1$, look at the quantities
	$$
	\phi_\ell := \frac{\beta_\ell}{\# J_\ell}
	$$
	for $\ell = 0, 1, ..., L+1$, where we set $J_0 := [n]$. Then $\phi_0 = \beta_0/n = k_3/n$. Since
	$$
	\frac{\phi_{\ell+1}}{\phi_\ell} = \frac{\beta_{\ell+1}}{\beta_\ell} \cdot \frac{\# J_\ell}{\# J_{\ell+1}} = \underbrace{\frac{1+\zeta\sqrt{\delta/\alpha_\ell}}{2}}_{\geq 1/2} \cdot \underbrace{\frac{\# J_\ell}{\# J_{\ell+1}}}_{\geq 2} \geq 1,
	$$
	we see that the $\phi_\ell$ are monotonically increasing. Thus
	$$
	\frac{k_3}{n} = \phi_0 \leq \phi_{L+1} = \frac{\beta_{L+1}}{\# J_{L+1}} \leq \frac{(2\zeta)^2\gamma\delta}{\# J_{L+1}},
	$$
	so that
	$$
	\# J = \# J_{L+1} \leq \frac{n}{k_3} \cdot (2\zeta)^2 \gamma \cdot k_1 \frac{m}{n} \leq 35.21 \cdot (2\zeta)^2 \frac{k_1}{k_2} \cdot m,
	$$
	i.e. $c_1 \leq 35.21 \cdot (2\zeta)^2k_1/k_2$.
\end{proof}
%

\section{Reproducing kernel Hilbert spaces}
\label{setting}
We will work in the framework of reproducing kernel Hilbert spaces. The relevant theoretical background can be found in \cite[Chapt.\ 1]{BeTh04} and \cite[Chapt.\ 4]{StChr08}. The papers \cite{HeBo04} and \cite{StSc12} are also of particular relevance for the subject of this paper.

Let $L_2(D,\varrho_D)$ be the space of complex-valued square-integrable functions with respect to~$\varrho_D$. Here $D \subset \R^d$ is an arbitrary measurable subset and $\varrho_D$ a measure on $D$.
We further consider a reproducing kernel Hilbert space $H(K)$ with a Hermitian positive definite kernel $K$ on $D \times D$. The crucial property of reproducing kernel Hilbert spaces is the fact that Dirac functionals are continuous, or, equivalently, the reproducing property
$$
		f(\bx) = \langle f, K(\cdot,\bx) \rangle_{H(K)}
$$
holds for all $\bx \in D$.

We will use the notation from \cite[Chapt.\ 4]{StChr08}. In the framework of this paper, the finite trace of the kernel is given by
\begin{equation}\label{integrab}
	\trace{K}:=\|K\|^2_{2} = \int_{D} K(\bx,\bx)d\varrho_D(\bx) < \infty \,.
\end{equation}
The embedding operator
\begin{equation}\label{f00b}
\Id_{K,\varrho_D}:H(K) \to L_2(D,\varrho_D)
\end{equation}
is Hilbert-Schmidt under the finite trace condition \eqref{integrab}, see \cite{HeBo04}, \cite[Lemma 2.3]{StSc12}, which we always assume from now on. We additionally assume that $H(K)$ is at least infinite dimensional. Let us denote the (at most) countable system of strictly positive eigenvalues $(\lambda_j)_{j\in \N}$ of $W_{K,\varrho_D} = \Id_{K,\varrho_D}^{\ast} \circ \Id_{K,\varrho_D}$ arranged in non-increasing order, i.e.,
$$
		\lambda_1 \geq \lambda_2 \geq \lambda_3 \geq \cdots > 0.
$$
We will also need the left and right singular vectors $(e_k)_k \subset H(K)$ and $(\eta_k)_k \subset L_2(D,\varrho_D)$ which both represent orthonormal systems in the respective spaces related by $e_k = \sigma_k \eta_k$ with $\lambda_k = \sigma_k^2$ for $k\in \N$\,. We would like to emphasize that the embedding \eqref{f00b} is not necessarily injective. In other words, for certain kernels there might also be a nontrivial null-space of the embedding in \eqref{f00b}. Therefore, the system $(e_k)_k$ from above is not necessarily a basis in $H(K)$. It would be a basis under additional restrictions, e.g.\ if the kernel $K(\cdot,\cdot)$ is continuous and bounded (i.e.\ a Mercer kernel).
It is shown in \cite{HeBo04}, \cite[Lemma 2.3]{StSc12} that if $\trace{K}<\infty$ and $H(K)$ is separable the non-negative function
\begin{equation}\label{almost_e}
	K(\bx,\bx) - \sum\limits_{k=1}^{\infty} |e_k(\bx)|^2	
\end{equation}
vanishes almost everywhere. Let us finally define the ``spectral functions''
\begin{equation}\label{f1b}
	N(m) := \sup\limits_{\bx \in D}\sum\limits_{k=1}^{m-1}|\eta_k(\bx)|^2\
\end{equation}
and
\begin{equation}\label{f1bb}
	T(m) := \sup\limits_{\bx \in D}\sum\limits_{k=m}^{\infty}|e_k(\bx)|^2
\end{equation}
provided that they exist.

\section{Weighted least squares}
\label{sect_lsqr}
Let us begin with concentration inequalities for the spectral norm of sums of complex rank-$1$ matrices. Such matrices appear as $\bL^\ast \bL$ when studying least squares solutions of over-determined linear systems
\begin{equation*} 
		\bL\cdot \mathbf{c} = \mathbf{f}\,,
\end{equation*}
where $\bL \in \C^{n \times m}$ is a matrix with $n>m$. It is well-known that the above system may not have a solution. However, we can ask for the vector $\mathbf{c}$ which minimizes the residual $\|\mathbf{f}-\bL\cdot \mathbf{c}\|_2$. Multiplying the system with $\bL^\ast$ gives
$$
	\bL^{\ast}\bL\cdot \mathbf{c} = \bL^{\ast}\cdot \mathbf{f}
$$
which is called the system of normal equations. If $\bL$ has full rank then the unique solution of the least squares problem is given by
\begin{equation}\label{ls}
	\mathbf{c} = (\bL^{\ast}\bL)^{-1}\bL^{\ast}\cdot \mathbf{f}\,.
\end{equation}
For function recovery problems we will use the following matrix
\begin{equation}\label{f0}
	\bL_{n,m} := \left(\begin{array}{cccc}
			\eta_1(\bx^1) &\eta_2(\bx^1)& \cdots & \eta_{m-1}(\bx^1)\\
			\vdots & \vdots && \vdots\\
			\eta_1(\bx^n) &\eta_2(\bx^n)& \cdots & \eta_{m-1}(\bx^n)
	\end{array}\right) =
	\left(\begin{array}{c}
	\by^1\\
	\vdots\\
	\by^n
	\end{array}\right)
\quad\text{and}\quad \mathbf{f} = 	\left(\begin{array}{c}
	f(\bx^1)\\
	\vdots\\
	f(\bx^n)
	\end{array}\right)\,,
\end{equation}
for $\bX = (\bx^1,...,\bx^n) \in D^n$ of distinct sampling nodes and a system $(\eta_k(\cdot))_{k}$ of functions. Here $\by^i := (\eta_1(\bx^i),...,\eta_{m-1}(\bx^i)), i=1,...,n$.

\begin{lemma}\label{prop2}\cite[Proposition 3.1]{KUV19} Let $\bL\in \C^{n\times m}$ be a matrix with $m\leq n$ with full rank and singular values  $\tau_1,...,\tau_m  >0$ arranged in non-increasing order.
\begin{description}
 \item[(i)] Then also the matrix $(\bL^{\ast}\bL)^{-1}\bL^{\ast}$ has full rank
and singular values $\tau_m^{-1},...,\tau_1^{-1}$ (arranged in non-increasing order).

\item[(ii)] In particular, it holds that
$$
   (\bL^{\ast}\bL)^{-1}\bL^{\ast} = \bV^{\ast} \tilde{\mathbf{\Sigma}}\bU
$$
whenever $\bL = \bU^{\ast} \mathbf{\Sigma}  \bV$, where $\mathbf{\Sigma} \in \re^{n\times m}$ is a rectangular matrix only with $(\tau_1,...,\tau_m)$ on the main diagonal and orthogonal matrices $\bU \in \C^{n\times n}$ and $\bV \in \C^{m\times m}$. Here $\tilde{\mathbf{\Sigma}} \in \re^{m\times n}$ denotes the matrix with $(\tau_1^{-1},...,\tau_m^{-1})$ on the main diagonal\,.

\item[(iii)] The operator norm $\|(\bL^{\ast}\bL)^{-1}\bL^{\ast}\|_{2 \to 2}$ can be controlled as follows
$$
	\tau_1^{-1} \leq \|(\bL^{\ast}\bL)^{-1}\bL^{\ast}\|_{2 \to 2} \leq \tau_m^{-1}\,.
$$
\end{description}
\end{lemma}

Being in the RKHS setting we compute the coefficients $c_k$, $k=1,\ldots,m-1$, of the approximant
\begin{equation}
S^m_{\bX}f:=\sum_{k = 1}^{m-1} c_k\, \eta_k\,\label{eq:def_SmX}
\end{equation}
using the least squares algorithm \eqref{ls}. We will also use the weighted version below, where
$\varrho_m(\cdot)$ is a density function which essentially first appeared in \cite{KrUl19} and has been adapted in \cite{MoUl20} to
\begin{equation}\label{density}
	\varrho_m(\bx) = \frac{1}{2} \bigg(\frac{1}{m-1} \sum_{j = 1}^{m-1} |\eta_j(\bx)|^2 + \frac{K(\bx,\bx) - \sum_{j = 1}^{m-1}|e_j(\bx)|^2}{\int_D K(\bx,\bx) d\varrho_D(\bx) - \sum_{j = 1}^{m-1}\lambda_j} \bigg).
\end{equation}

\begin{algorithm}[H]
\caption{Weighted least squares approximation \cite{CoMi17},\cite{KrUl19},\cite{KUV19}.}\label{algo1:reweighted}
  \begin{tabular}{p{1.2cm}p{4.5cm}p{8.9cm}}
    Input: & $\bX = (\bx^1,...,\bx^n)\in D^n$ \hfill & matrix of distinct sampling nodes, \\
      & $\mathbf{f} = (f(\bx^1),...,f(\bx^n))^\top$ \hfill & samples of $f$ evaluated at the nodes from $\bX$, \\
      & $m\in\N$ & $m \leq n$ such that the matrix $\tilde{\bL}_{n,m}$ in \eqref{eq:tilde_L} has full (column) rank.
  \end{tabular}
  \begin{algorithmic}
	\STATE
      Compute weighted samples $\boldsymbol{g}:=(g_j)_{j=1}^n$ with $g_j:=\begin{cases}0,  & \varrho_m(\bx^j)=0,\\
      f(\bx^j)/\sqrt{\varrho_m(\bx^j)}, & \varrho_m(\bx^j)\neq 0\,.
       \end{cases}$

  	\STATE
  	Solve the over-determined linear system
  	\begin{equation}
  	\widetilde{\bL}_{n,m} \cdot (\tilde{c}_1,...,\tilde{c}_{m-1})^\top = \mathbf{g}\,, \; \widetilde{\bL}_{n,m}:=\Big(l_{j,\ell}\Big)_{j=1,\ell=1}^{n,m-1},\; l_{j,\ell}:=\begin{cases}0,  & \varrho_m(\bx^j)=0,\\
  	      \eta_{\ell}(\bx^j)/\sqrt{\varrho_m(\bx^j)}, & \varrho_m(\bx^j)\neq 0,
  	       \end{cases}
  	       \label{eq:tilde_L}
  	\end{equation}
  	via least squares (e.g.\ directly or via the LSQR algorithm \cite{PaSa82}), i.e., compute
  	$$
  	(\tilde{c}_1,...,\tilde{c}_{m-1})^\top := (\widetilde{\bL}_{n,m}^{\ast}\widetilde{\bL}_{n,m})^{-1} \,\widetilde{\bL}_{n,m}^{\ast}\cdot \mathbf{g}.
  	$$
  \end{algorithmic}
   Output:  $\mathbf{\tilde{c}} = (\tilde{c}_1,...,\tilde{c}_{m-1})^\top\in \C^{m-1}$ coefficients of the approximant $\widetilde{S}_{\bX}^m f:=\sum_{k = 1}^{m-1} \tilde{c}_{k} \eta_{k}$.
\end{algorithm}
Note, that the mapping $f \mapsto \widetilde{S}^m_{\bX}f$ is well-defined and linear for a fixed set of sampling nodes $$\bX = (\bx^1,...,\bx^n) \in D^n$$
if the matrix $\widetilde{\bL}_{n,m}$ has full (column) rank. The next section gives sufficient conditions when this is the case.

\section{Concentration results for random matrices}
\label{sect_prob}

We start with a concentration inequality for the spectral norm of a matrix of type \eqref{f0}. It turns out that the complex matrix $\bL_{n,m}:=\bL_{n,m}(\bX)\in\C^{n\times(m-1)}$ has full rank with high probability, if $\bX = (\bx^1,...,\bx^n)$ is drawn at random from $D^n$ according to a measure $\Prob = d\varrho^n$, the functions $(\eta_k)_{k=1}^{m-1}$ are orthonormal w.r.t the measure $\varrho$ and $m$ is not too large (compared to $n$). We will find below that the eigenvalues of
\begin{equation} \label{defHm}
\mathbf{H}_m:=\mathbf{H}_m(\bX) = \frac{1}{n}\bL_{n,m}^\ast \bL_{n,m}
= \frac{1}{n} \sum\limits_{i = 1}^n \by^i \otimes \by^i \in\C^{(m-1)\times(m-1)}
\end{equation}
are bounded away from zero with high probability if $m$ is small enough compared to $n$. The following result is a consequence of \cite[Thm.\ 1.1]{Tr11}, see also \cite[Thm.\ 2.3, Cor.\ 2.5]{MoUl20}.

\begin{theorem} \label{control}
For \(n \geq m\), \(r > 1\) we immediately obtain that the matrix \(\mathbf{H}_m\)  has only eigenvalues greater than $ 1/2$ and smaller than $3/2$ with probability at least \(1 - 2n^{1-r}\) if
the nodes are sampled i.i.d.\ according to $\varrho$ and
\begin{equation}\label{f20}
 N(m) \leq \frac{n}{ 10 \,  r \log n},
\end{equation}
where $N(m)$ is the quantity defined in \eqref{f1b}\,.
Equivalently, we have for all $\bw \in \C^{m-1}$
$$
	\frac{1}{2}\|\bw\|_2^2 \leq \frac{1}{n}\big\|\bL_{n,m} \bw\big\|_2^2 \leq \frac{3}{2}\|\bw\|_2^2
$$
and
\begin{equation}\label{f100}
   \sqrt{\frac{2}{3n}} \leq \|(\bL_{n,m}^{\ast}\bL_{n,m})^{-1}\bL_{n,m}^{\ast}\|_{2\to 2} \leq \sqrt{\frac{2}{n}}
\end{equation}
with probability at least $1-2n^{1-r}$\,.

\end{theorem}

Let us now turn to infinite matrices. We need a result which can be applied to independent $\ell_2$-sequences of the form
$$
	\by^i = (e_m(\bx^i),e_{m+1}(\bx^i),...)\quad,\quad i=1,...,n\,,
$$
where $(e_k)_k \subset H(K)$ is the system of right singular vectors of the embedding $\Id:H(K) \to L_2$ defined above.

The following infinite-dimensional concentration result is proved in \cite[Thm.\ 1.1]{MoUl20}. There are earlier versions for the finite-dimensional framework (matrices) proved by Tropp \cite{Tr11}, Oliveira \cite{Ol10}, Rauhut \cite{Ra10} and others. Mendelson, Pajor \cite{PaMe06} and also  Oliveira \cite{Ol10} comment on infinite versions of their result. The key feature of the following proposition is the exact control of the constants and the decaying fail probability, see also Remark 3.10 in \cite{MoUl20} for a more detailed comparison to earlier results.

\begin{proposition} \label{main1}
Let \(\by^i , i= 1 \dots n \), be i.i.d random sequences from \( \ell_2\). Let further $n \geq 3$, $r > 1$, $M>0$ such that \(\| \by^i \|_2 \leq M\) for all $i=1 \dots n$ almost surely and \( \Ept  \by^i \otimes  \by^i  =  {\bf \Lambda}\)  for all \(i=1 \dots n\). Then
\[ \Prob \Big( \Big\| \frac{1}{n} \sum_{i=1}^n  \, \by^i \otimes  \by^i - {\bf \Lambda} \Big\|_{2 \to 2} \geq F \Big) \leq 2^\frac{3}{4} \, n^{1-r}\,,\]
where \(F := \max\Big\{ \frac{8r \log n}{n} M^2 \kappa^2 , \| {\bf \Lambda} \|_{2 \to 2} \Big\}\) and \( \kappa = \frac{1+\sqrt{5}}{2}\) .
\end{proposition}

This can be written in a more compact form.
\begin{theorem}\label{cor} Let \(\by^i , i= 1 \dots n, \) be i.i.d random sequences from \( \ell_2\). Let further $n \geq 3$, $M>0$ such that \(\| \by^i \|_2 \leq M\) for all $i=1 \dots n$ almost surely and \( \Ept  \by^i \otimes  \by^i  =  {\bf \Lambda}\) for $i=1,...,n$ with $\|\boldsymbol{\Lambda}\|_{2\to 2} \leq 1$. Then, for $0<t<1$,
\[ \Prob \Big( \Big\| \frac{1}{n} \sum_{i=1}^n  \, \by^i \otimes  \by^i - {\bf \Lambda} \Big\|_{2 \to 2} \geq t \Big) \leq 2^\frac{3}{4} n\exp\Big(-\frac{t^2n}{21M^2}\Big)\,.\]
\end{theorem}

\section{New bounds for sampling numbers}
We are interested in the question of optimal sampling recovery of functions from reproducing kernel Hilbert spaces in $L_2(D,\varrho_D)$. The quantity we want to study is classically given by
$$
	g_n(\Id_{K,\varrho_D}) := \inf\limits_{\bx^1,...,\bx^n \in D} \; \inf\limits_{\varphi:\C^n \to L_2} \; \sup\limits_{\|f\|_{H(K)}\leq 1}\|f-\varphi(f(\bx^1),...,f(\bx^n))\|_{L_2(D,\varrho_D)}
$$
and quantifies the recovery of functions out of $n$ function values in the worst case setting.
The goal is to get reasonable bounds for this quantity in $n$, preferably in terms of the singular numbers of the embedding. Results on the decay properties of this quantity in the framework of RKHS have been given by several authors, see, e.g., \cite{KuWaWo09}, \cite{KrUl19}, \cite{NoWoIII} and the references therein (see also Remark \ref{MoeUl}). For a special case in the field of Hyperbolic Cross Approximation we refer to  \cite[Outstanding Open Problem 1.4]{DuTeUl19}. Here we present a new upper bound in the general framework.

\paragraph{The main idea.} In the following theorem we apply Weaver's theorem to a random frame. The idea is to construct a sampling operator $\widetilde{S}^m_J$ using $O(m)$ sampling nodes as follows. We draw nodes $\bX = (\bx^1,...,\bx^n)$ i.i.d.\ at random according to some measure $\mu_m$ specified concretely in the proof below, where $n$ scales as $m\log m$. At this stage we have too many sampling nodes, however, a ``good'' frame
in the sense of a well-conditioned matrix $\bL_{n,m}$ in~\eqref{f0}. To this frame (rows of $\bL_{n,m}$) we apply our modified Weaver theorem. The result is a shrinked well-conditioned sub-frame corresponding to a subset $(\bx^i)_{i\in J}$ of the initial set of sample nodes. With this sub-frame (sub-matrix of $\bL_{n,m}$) we solve the over-determined system via the least squares Algorithm 1. This represents the sampling operator. When it comes to the error analysis we again benefit from the fact that we deal with a (not too small compared to $n$) subset of the original nodes $\bX$. The consequence is that we do not pay too much ($\sqrt{\log n}$) compared to the sampling operator based on the original set $\bX$ of nodes. However, we only used $O(m)$ sample nodes which makes the difference.

\begin{theorem}\label{thm9} Let $H(K)$ be a separable reproducing kernel Hilbert space on a set $D \subset \R^d$ with a positive semidefinite kernel $K:D\times D\to\C$ satisfying
$$
	\int_D K(\bx,\bx) d\varrho_D(\bx) < \infty
$$
for some measure $\varrho_D$ on $D$. Then $\Id_{K,\varrho_D}:H(K) \to L_2(D,\varrho_D)$ is a Hilbert-Schmidt embedding, the corresponding
sequence of singular numbers $(\sigma_k)_{k=1}^{\infty}$ square-summable.
For the sequence of sampling numbers $g_n:=g_n(\Id_{K,\varrho_D})$ we have the general bound
\begin{equation}\label{samplingnum}
	g_n^2 \leq C\frac{\log n}{n}\sum\limits_{k\geq cn} \sigma_k^2  \quad,\quad n\geq 2\,,
\end{equation}
with two universal constants $C,c>0$, which are specified in Remark \ref{constants} below.
\end{theorem}

\begin{proof}\sloppy
Let $m \geq 2$. Similar as in \cite{KrUl19}, \cite{KUV19} and \cite{MoUl20} we use the density function \eqref{density} in order to consider the embedding $\widetilde{\Id}:H(\widetilde{K}_m) \to L_2(D,\varrho_m(\cdot)d\varrho_D)$ instead of $\Id_{K,\varrho_D}:H(K) \to L_2(D,\varrho_D)$. In fact, we define the new kernel
\begin{equation}\label{tilde_kernel}
	\widetilde{K}_m(\bx,\by) := \frac{K(\bx,\by)}{\sqrt{\varrho_m(\bx)}\sqrt{\varrho_m(\by)}}\,.
\end{equation}
This yields
\begin{equation}\label{f_to_g}
\sup_{\|f\|_{H(K)} \leq 1 } \big\| f - \widetilde{S}_{\bX}^m f \big\|_{L_2(D,\varrho_D)} \leq \sup_{\|g\|_{H(\widetilde{K}_m)} \leq 1 } \big\| g - S_{\bX}^m g \big\|_{L_2(D,\varrho_m(\cdot)d\varrho_D)}\,
\end{equation}
and $\widetilde{N}(m) \leq 2(m-1)$, $\widetilde{T}(m) \leq 2\sum_{j=m}^{\infty}\sigma_j^2$. For the details of this, see the discussion in the proof of \cite[Thm.\ 5.9]{KUV19} and \cite[Thm.\ 5.5]{KUV19}. Note, that the operators $\widetilde{S}_{\bX}^m$ and $S_{\bX}^m$ are defined by Algorithm 1 and \eqref{ls}. The number of samples will be chosen later as $O(m)$. Choose now the smallest $n$ such that
$$
	m \leq \frac{n}{40\log(n)}\,.
$$
This implies $\widetilde{N}(m) \leq 2(m-1) \leq n/(20 \log(n))$\,. Applying Theorem~\ref{control} with $r=2$ gives that the rows of the matrix $\frac{1}{\sqrt{n}}\widetilde{\bL}_{n,m}$ represent a finite frame with frame bounds $1/2$ and $3/2$ with high probability (the failure probability $2n^{-1}$ decays polynomially in $n$) when $\bX = (\bx^1,...,\bx^n) \in D^n$ is sampled w.r.t to the measure $d\mu_m = \varrho_m(\cdot)d\varrho_D$. That means, we have with high probability for any $\bw \in \C^{m-1}$
\begin{align}\label{fraspec}
\frac{1}{2}\|\bw\|_2^2 \leq \frac{1}{n}\big\|\widetilde{\bL}_{n,m} \bw\big\|_2^2 \leq \frac{3}{2}\|\bw\|_2^2 \,.
\end{align}

Let us now denote with $P_{m-1}:H(\widetilde{K}_m) \to H(\widetilde{K}_m)$ the projection operator onto $\mbox{span}\{e_1(\cdot)/\sqrt{\varrho_m(\cdot)},...,e_{m-1}(\cdot)/\sqrt{\varrho_m(\cdot)}\}$. Following the proof in \cite[Thm.\ 5.1]{MoUl20}, we obtain
almost surely
\begin{equation}
\begin{split}
\sup\limits_{\|g\|_{H(\widetilde{K}_m)}\leq 1}\frac{1}{n}\sum\limits_{i=1}^n|g(\bx^i)-P_{m-1}g(\bx^i)|^2 &\leq \frac{1}{n}\|\mathbf{\Phi}_m^{\ast}\mathbf{\Phi}_m\|_{2\to 2}  \\
&\leq \Big\|\frac{1}{n}\mathbf{\Phi}_m^{\ast}\mathbf{\Phi}_m-\mathbf{\Lambda}\Big\|_{2\to 2} + \|\mathbf{\Lambda}\|_{2\to 2}
\end{split}
\end{equation}
with $\mathbf{\Lambda} = \diag(\sigma_m^2,\sigma_{m+1}^2,...)$ and
$$
   \mathbf{\Phi}_m:=\left(\begin{array}{c}
			\by^1\\
			\vdots\\
			\by^n
		 \end{array}\right)\,
$$
being infinite matrices/operators with sequences \(\by^i = 1/\sqrt{\varrho_m(\bx^i)}\left(e_m(\bx^i),e_{m+1}(\bx^i), \dots\right)\), \( i=1 \dots n\). By
Proposition \ref{main1} we finally get from this with high probability
\begin{equation}\label{bound}
   \sup\limits_{\|g\|_{H(\widetilde{K}_m)}\leq 1}\frac{1}{n}\sum\limits_{i=1}^n|g(\bx^i)-P_{m-1}g(\bx^i)|^2  \leq c_1\Big(\sigma_m^2 + \frac{1}{m}\sum\limits_{k=m}^\infty \sigma_k^2\Big)\,,
\end{equation}
for some constant $c_1>0$, where we used that
$$
\|\by^i\|_2^2\le\widetilde{T}(m) \leq 2\sum_{j=m}^{\infty} \sigma_j^2 \,.
$$
Due to the high probability of both events, \eqref{fraspec} and \eqref{bound}, there exists an instance of nodes $\{\bx^1,...,\bx^n\}$ such that the bound \eqref{bound} is true and $\frac{1}{\sqrt{n}}\widetilde{\bL}_{n,m}$ represents a finite frame in $\C^{m-1}$.  Note that all the assumptions of Theorem~\ref{thm_weaver} are fulfilled. Indeed, the squared Euclidean norms of the rows of
$\frac{1}{\sqrt{n}}\widetilde{\bL}_{n,m}$ are bounded by $2(m-1)/n$. To this finite frame we may thus apply Theorem \ref{thm_weaver} which constructs a sub-matrix $\widetilde{\bL}_{J,m}$ having $\# J = O(m)$ rows which, properly normalized, still form a frame in $\C^{m-1}$\,. It holds for all $\bw \in \C^{m-1}$
\begin{equation}\label{shrink}
c_2\|\bw\|_2^2 \leq \frac{1}{m}\|\widetilde{\bL}_{J,m}\bw\|_2^2 \leq C_2\|\bw\|_2^2\,.
\end{equation}
With this matrix we perform the least squares method \eqref{ls} applied to the shrinked vector of function samples $(f(\bx^i))_{i\in J}$ corresponding to the rows of $\widetilde{\bL}_{J,m}$. We denote the least squares operator with $S^m_J$. Note, that this operator uses only $\# J = O(m)$ samples. Let $\|g\|_{H(\widetilde{K}_m)}\leq 1$ and again denote with $P_{m-1}:H(\widetilde{K}_m) \to H(\widetilde{K}_m)$ the projection operator onto $\mbox{span}\{e_1(\cdot)/\sqrt{\varrho_m(\cdot)},...,e_{m-1}(\cdot)/\sqrt{\varrho_m(\cdot)}\}$. Then it holds
\begin{equation}
  \begin{split}
  	\|g-S^m_Jg\|^2_{L_2} &= \|g-P_{m-1}g + P_{m-1}g -S^m_Jg\|^2_{L_2}\\
  	&= \|g-P_{m-1}g\|_{L_2}^2 + \|S^m_J(P_{m-1}g -g)\|^2_{L_2}\\
  	&\leq \sigma_m^2 + \|(\widetilde{\bL}_{J,m}^\ast\widetilde{\bL}_{J,m})^{-1}\widetilde{\bL}_{J,m}^{\ast}\|^2_{2\to 2}\sum\limits_{i\in J}|g(\bx^i)-P_{m-1}g(\bx^i)|^2\,.
 \end{split}
\end{equation}
Using Lemma \ref{prop2} together with \eqref{shrink} gives
$$	\|g-S^m_Jg\|^2_{L_2} \leq \sigma_m^2+\frac{c_3}{m}\sum\limits_{i=1}^n |g(\bx^i) - P_{m-1}g(\bx^i)|^2 \,.
$$	
By the choice of $n$ together with \eqref{bound} we may estimate further
\begin{equation}
	\begin{split}
	\|g-S^m_Jg\|^2_{L_2} &\leq c_4\log(n)\Big(\sigma_m^2 + \frac{1}{m}\sum\limits_{k=m}^\infty \sigma_k^2\Big)\\
	 &\leq c_5\frac{\log(m)}{m}\sum\limits_{k=\lfloor 	m/2 \rfloor}^\infty \sigma_k^2\,.
  \end{split}
\end{equation}
Consequently, by \eqref{f_to_g} we obtain
$$
  \sup_{\|f\|_{H(K)} \leq 1 } \big\| f - \widetilde{S}_{J}^m f \big\|_{L_2(D,\varrho_D)} \leq c_5\frac{\log(m)}{m}\sum\limits_{k=\lfloor 	m/2 \rfloor}^\infty \sigma_k^2\ 	
$$
and finally, using that $\# J = O(m)$,
$$
	g_{\lceil c_6 m \rceil}^2 \leq c_5\frac{\log(m)}{m}\sum\limits_{k=\lfloor m/2 \rfloor}^\infty \sigma_k^2 \,,
$$
where all the involved constants are universal. This implies the statement of the theorem.
\end{proof}

\begin{remark}\label{MoeUl}{\em (i)} The additional assumption on the ``separability'' of $H(K)$ ensures the equality sign in the inequality
\begin{equation}\label{f98}
\sum\limits_{k=1}^{\infty} \sigma_k^2 \leq \int_D K(\bx,\bx)d\varrho_D(\bx) < \infty.
\end{equation}
The identity is crucial in the proof of Theorem \ref{thm9}, see also \cite[Thm.\ 5.5]{KUV19}. Without an equality in \eqref{f98} the best known general upper bound for the sampling numbers $g_n$ can be found in the recent paper \cite{MoUl20}. In fact, combining the proof of Theorem \ref{thm9} with the one in \cite[Thm.\ 7.1]{MoUl20} one can prove
$$
	g_n^2 \leq C\max\Big\{\frac{1}{n},\frac{\log n}{n}\sum\limits_{k\geq cn} \sigma_k^2\Big\}\,.
$$
The bound is worse compared to \eqref{samplingnum}. However, one should rather compare to the bound in \cite{WaWo01}, since this proof also works without equality in \eqref{f98}. There the authors proved under the same assumptions
$$
	g_n^2 \leq \min\Big\{\sigma_\ell^2  + \frac{\trace{K}\ell}{n}~:~\ell=1,2,...\Big\}.
$$

{\em (ii)} As far as we know, G. Wasilkowski and H. Wo{\'z}niakowski \cite{WaWo01} (2001) were the first who addressed the sampling recovery problem in 2001 in the context of reproducing kernel Hilbert spaces. They obtained results using exclusively the finite trace condition \eqref{trace}. Later in 2009, F. Kuo, G. Wasilkowski and H. Wo{\'z}niakowski \cite{KuWaWo09} did further progress by determining the range for the ``power of standard information'' when knowing the decay rate of the singular numbers. We refer to the monograph \cite{NoWoIII} for a detailed historical discussion of the development until 2012. The recent progress in the field has been initiated by D. Krieg and M. Ullrich \cite{KrUl19} in 2019. The authors proved under stronger assumptions than in Theorem \ref{thm9} an existence result, namely
$$
	g_n^2 \lesssim \frac{\log n}{n}\sum\limits_{k \gtrsim n/\log(n)} \sigma_k^2\,,
$$
which represents a slightly weaker bound. Further progress (explicit constants, consequences for numerical integration) has been given in L. K\"ammerer, T. Ullrich, T. Volkmer \cite{KUV19} and in M. Ullrich \cite{Ul20} as well as M. Moeller, T. Ullrich \cite{MoUl20} for the control of the failure probability.

{\em (iii)} To further demonstrate the usefulness of the tools developed here, let us mention that they have already been used by others. In fact, while this manuscript was under review, D. Krieg and M. Ullrich \cite{KrUl20} used this technique to observe that for non-Hilbert function space embeddings the ``gap'' is also at most $\sqrt{\log(n)}$ if the corresponding approximation numbers are $p$-summable with $p<2$. Finally, we would like to mention a very recent result by V.N. Temlyakov \cite{Te20_2}, where the sampling numbers $g_n$ of a function class $\mathbf{F}$ in $L_2$ are related to the Kolmogorov numbers in $L_\infty$. This allows for treating the case of ``small smoothness'', where the square summability does not hold \cite{TeUl20_1, TeUl20_2}.

\end{remark}

\begin{remark}\label{constants} Note that, using Theorem \ref{thm_weaver}, Theorem \ref{control} and Proposition \ref{main1}, we can get explicit values for the constants in the above theorem.
Concretely, we can conclude that \eqref{samplingnum} holds for $m\ge13136$
with $C=1.5 \cdot 10^6 $ and $c=3.8 \cdot 10^{-5}$.

To verify this, let $m \geq 2$ and $n=n(m)$, as in the proof of Theorem~\ref{thm9}, such that
$$
\frac{n-1}{40 \log(n-1)} < m \leq \frac{n}{40 \log(n)}.
$$
For $m=2$ we get $n=497$ and since $n$ increases as $m$ increases, we generally have $n \geq 497$. Put $r=2$. We then also have \eqref{fraspec} with probability at least $1-2/n \geq 1-2/497$ and \eqref{bound} with probability at least $1-2^{3/4}/n \geq 1-2^{3/4}/497$. Since these already sum up to $ (1-2/497) + (1-2^{3/4}/497) > 1$, we can guarantee the existence of a node set $\{\textbf{x}^1, \dots, \textbf{x}^n\}$ as in the proof of Theorem \ref{thm9}. \\
In the proof, we apply Theorem \ref{thm_weaver} with $k_1 = 2, k_2 = 1/2, k_3 = 3/2$, so that we get the respective constants
$$
\tilde{c}_1 = 6568, \quad \tilde{c}_2 = 2(2+\sqrt{2})^2, \quad \tilde{c}_3 = 9852.
$$
For this, note that for all $m\ge2$ and $n=n(m)$ as above we have
$$
\frac{n}{m} \geq 40\log(n) \geq 40\log(497) \geq 40\cdot 4.7 = 47 \frac{k_1}{k_2} \,.
$$
To calculate $c_1$, we apply Proposition \ref{main1} with $M^2 = 2 \sum_{k=m}^\infty \sigma_k^2$, so that we get
$$
F = \max\left\{\sigma_m^2, \frac{16 \log(n)}{n} M^2 \kappa^2\right\} = \max\left\{\sigma_m^2, 32 \left(\frac{1+\sqrt{5}}{2}\right)^2 \frac{\log(n)}{n} \sum_{k=m}^\infty \sigma_k^2\right\}.
$$
Using $\log(n)/n \leq 1/40m$, we can then estimate further
$$
F \leq \max\left\{\sigma_m^2, \frac{32}{40} \left(\frac{1+\sqrt{5}}{2}\right)^2 \frac{1}{m} \sum_{k=m}^\infty \sigma_k^2\right\} \leq \underbrace{\left(1+\frac{4}{5}\left(\frac{1+\sqrt{5}}{2}\right)^2\right)}_{=: c_1} \left(\sigma_m^2 + \frac{1}{m} \sum_{k=m}^\infty \sigma_k^2\right),
$$
so that $c_1 = 3.09\dots$. \\
After applying Theorem \ref{thm_weaver}, we get $c_2 = \tilde{c}_2 = 2(2+\sqrt{2})^2$ and $C_2 = \tilde{c}_3 = 9852$, as well as $\# J \leq 6568m$. \\
Since
$$
mc_2 \|\textbf{w}\|_2^2 \leq \|\widetilde{\textbf{L}}_{J, m} \textbf{w}\|_2^2 \leq C_2m\|\textbf{w}\|_2^2,
$$
Lemma \ref{prop2} (iii) gives
$$
\frac{1}{C_2m} \leq \|(\widetilde{\textbf{L}}_{J, m}^* \widetilde{\textbf{L}}_{J, m})^{-1} \widetilde{\textbf{L}}_{J, m}^*\|_{2 \to 2}^2 \leq \frac{1}{c_2m},
$$
so that we may choose $c_3 = 1/c_2 = (2+\sqrt{2})^{-2}/2$. \\
To obtain $c_4$, we use $n/m \leq \frac{n}{n-1} \cdot 40 \log(n-1) \leq \frac{497\log(496)}{496\log(497)} \cdot 40 \log(n)$ to estimate
\begin{align*}
\|g-\widetilde{S}^m_Jg\|^2_{L_2} & \leq \sigma_m^2+\frac{c_3}{m}\sum\limits_{i=1}^n |g(\bx^i) - P_{m-1}g(\bx^i)|^2 \\
 & \leq \sigma_m^2+ c_3 \frac{n}{m} \cdot c_1\left(\sigma_m^2 + \frac{1}{m}\sum\limits_{k=m}^\infty \sigma_k^2\right) \\
 & \leq \sigma_m^2 + 40 \cdot \frac{497\log(496)}{496\log(497)} c_1c_3 \log(n)\left(\sigma_m^2 + \frac{1}{m}\sum\limits_{k=m}^\infty \sigma_k^2\right) \\
 & \leq \underbrace{\left(1+40 \cdot \frac{497\log(496)}{496\log(497)} c_1c_3\right)}_{=: c_4} \log(n) \left(\sigma_m^2 + \frac{1}{m}\sum\limits_{k=m}^\infty \sigma_k^2\right)
\end{align*}
and get $c_4 = 6.31\dots$. \\
As for $c_5$, we start with
$$
\sigma_m^2 \leq \frac{1}{(m-1)-\lfloor m/2\rfloor +1} \sum_{k=\lfloor m/2\rfloor}^{m-1} \sigma_k^2 = \frac{1}{\lceil m/2\rceil} \sum_{k=\lfloor m/2\rfloor}^{m-1} \sigma_k^2 \leq \frac{2}{m} \sum_{k=\lfloor m/2\rfloor}^{m-1} \sigma_k^2,
$$
so that
$$
\sigma_m^2 + \frac{1}{m} \sum_{k=m}^\infty \sigma_k^2 \leq \frac{2}{m} \sum_{k=\lfloor m/2\rfloor}^\infty \sigma_k^2.
$$
Furthermore, we have $\log(n-1)-\log(40)-\log \log (n-1) < \log (m)$ and
$$
\frac{\log(n-1)-\log(40)-\log \log (n-1)}{\log(n)} \geq \frac{\log(496)-\log(40)-\log \log (496)}{\log(497)} =: \vartheta = 0.11\dots,
$$
from which we conclude $\log(n) \leq \vartheta^{-1} \log(m)$. We then have
$$
c_4 \log(n) \left(\sigma_m^2 + \frac{1}{m}\sum\limits_{k=m}^\infty \sigma_k^2\right) \leq \underbrace{\frac{2c_4}{\vartheta}}_{=: c_5} \cdot \frac{\log(m)}{m} \sum_{k=\lfloor m/2\rfloor}^{\infty} \sigma_k^2
$$
with $c_5 = 113.35\dots$. \\
Finally observe that we can choose $c_6=\tilde{c}_1=6568$ due to $\# J \leq \tilde{c}_1m=6568m$.\\
Now take $\tilde{m}\ge 2\tilde{c}_1$ and $m=\lfloor \tilde{m}/\tilde{c}_1 \rfloor\ge 2$. Further note that $\lfloor \tilde{m}/2\tilde{c}_1 \rfloor = \lfloor\lfloor \tilde{m}/\tilde{c}_1 \rfloor /2\rfloor$. Then
$$
g_{\tilde{m}}^2 \leq g_{\tilde{c}_1 m}^2 \leq c_5 \frac{\log(m)}{m}\sum\limits_{k=\lfloor m/2 \rfloor}^\infty \sigma_k^2 = c_5 \frac{\log(\lfloor \tilde{m}/\tilde{c}_1 \rfloor)}{\lfloor \tilde{m}/\tilde{c}_1 \rfloor}\sum\limits_{k=\lfloor \tilde{m}/(2\tilde{c}_1) \rfloor}^\infty \sigma_k^2 \,.
$$
The asserted estimates now follow due to $2\tilde{c}_1=13136$, $2 c_5 \tilde{c}_1 \leq 1.5 \cdot 10^6$, and $1/(4\tilde{c}_1) \geq 3.8 \cdot 10^{-5}$.
\end{remark}

\begin{remark}\label{rem63} The interesting question remains, whether there is a situation where the above bound on sampling numbers is sharp. Let us refer to the next subsection for a possible candidate. Clearly, there are situations where the bound in Theorem \ref{thm9} does not reflect the correct behavior of sampling numbers. This is for instance the case for the univariate Sobolev embedding $\Id:H^1([0,1]) \to L_2([0,1])$ where the sampling numbers show, at least asymptotically, the same behavior as the singular numbers.
\end{remark}

\section{An outstanding open problem}
\label{discussion}

Let us once again comment on an important open problem for the optimal sampling recovery of multivariate functions. We consider the minimal worst-case error (sampling numbers/widths) defined by
\begin{equation}\label{s_mix}
	g_n(\Id_{s,d}:H^s_{\text{mix}}(\tor^d) \to L_2(\tor^d)) := \inf\limits_{\bX = (\bx^1,...,\bx^n)}\inf\limits_{\varphi:\C^n \to L_2} \sup\limits_{\|f\|_{H^s_{\text{mix}}}\leq 1}\|f-\varphi(f(\bX))\|_{L_2(\tor^d)}\,.
\end{equation}
Let us comment on the class $H^s_{\text{mix}}(\tor^d)$. That is, we consider functions on the $d$-dimensional torus $\tor^d \simeq [0,1)^d$, where $\tor$ stands for $[0,1]$ with endpoints identified. Note, that the unit cube $[0,1]^d$ is preferred here since it has Lebesgue measure $1$ and is therefore a probability space. We could have also worked with $[0,2\pi]^d$ and the Lebesgue measure (which can be made a probability measure by a $d$-dependent rescaling).

There are many different ways to define function spaces of dominating mixed soothness, see \cite[Chapt.\ 3]{DuTeUl19}. We choose an approach which is closely related to \cite[Sect.\ 2.1]{KSU2}, see also (2.6) there. In fact, $L_2(\T^d)$-norms of mixed derivatives of the multivariate function $f$ can be written in terms of Fourier coefficients $\hat{f}_{\bk}$ of $f$. For $\alpha\in \N$ we define the space $H^{\alpha}_{\textnormal{mix}}(\tor^d)$ as the Hilbert space with the inner product
\begin{equation}\label{equ:inner_product_star}
	\langle f , g\rangle_{H^{\alpha}_{\textnormal{mix}}} := 	\sum\limits_{\bj \in \{0,\alpha\}^d}\langle D^{(\bj)}f,D^{(\bj)}g \rangle_{L_2(\tor^d)}\,.
\end{equation}
$D^{(j_1,\ldots,j_d)}=\partial_1^{j_1}\cdots \partial_d^{j_d}$ thereby denotes the weak derivative operator.
Defining the weight
\begin{equation}\label{norm2}
	w_{\alpha}(k) = (1+(2\pi |k|)^{2\alpha})^{1/2}\quad,\quad k\in \Z\,
\end{equation}
and the univariate kernel function
\begin{equation}\nonumber
	K^1_{\alpha}(x,y):=\sum\limits_{k\in \Z} \frac{\exp(2\pi\mathrm{i}k(y-x) )}{w_{\alpha}(k)^2}\quad,\quad x,y\in \tor\,,
\end{equation}
directly leads to
\begin{equation}\label{kerneld}
		K^d_{\alpha}(\bx,\by):=	K^1_{\alpha}(x_1,y_1)\cdots K^1_{\alpha}(x_d,y_d)\quad,\quad \bx,\by \in \tor^d\,,
\end{equation}
which is a reproducing kernel for $H^\alpha_{\textnormal{mix}}(\tor^d)$. In particular, for any $f\in H^\alpha_{\textnormal{mix}}(\tor^d)$ we have
$$
	f(\bx) = \langle f, K^d_{\alpha}(\bx,\cdot) \rangle_{H^{\alpha}_{\textnormal{mix}}}\,.
$$
The kernel defined in \eqref{kerneld} associated to the inner product \eqref{equ:inner_product_star} can be extended to the case of fractional smoothness $s>0$ replacing $\alpha$ by $s$ in \eqref{norm2}--\eqref{kerneld} which in turn leads to the inner product
\begin{equation*}%
 \langle f , g\rangle_{H^{s}_{\textnormal{mix}}} := \sum\limits_{\bk\in \Z^d} \hat{f}_{\bk} \, \overline{\hat{g}_{\bk}} \, \prod_{j=1}^d w_{s}(k_j)^2\,
\end{equation*}
in terms of the Fourier coefficients $\hat{f}_{\bk}$, $\hat{g}_{\bk}$ and the corresponding norm. The (ordered) sequence $(\lambda_j)_{j=1}^{\infty}$ of eigenvalues of the corresponding mapping $W_{s,d} = \Id_{s,d}^{\ast}\circ \Id_{s,d}$, where $\Id\colon \mathrm{H}\big(K^d_{s}\big) \to L_2(\tor^d)$, is the non-increasing rearrangement of the numbers
\begin{equation*}
	\Big\{\lambda_{\bk}:=\prod\limits_{j=1}^d w_{s}(k_j)^2=\prod\limits_{j=1}^d (1+(2\pi |k_j|)^{2s})^{-1}~:~\bk \in \Z^d\Big\}\,.
\end{equation*}
It has been shown by various authors, see \cite[Chapt.\ 4]{DuTeUl19} and the references therein, that we have asymptotically ($s>0$)
\begin{equation}\label{sigma_n}
	\sigma_n(\Id_{s,d}) \asymp_{s,d} n^{-s}(\log n)^{(d-1)s}\quad,\quad n\geq 2\,.
\end{equation}

The correct asymptotic behavior of \eqref{s_mix} has been addressed by several authors in the Information Based Complexity (IBC) community, see, e.g., \cite{NoWoIII} and also \cite[Outstanding Open Problem 1.4]{DuTeUl19}. It is  nowadays well-known, see e.g.\ \cite{SiUl07}, \cite{Du11}, \cite{Tr10} and \cite[Sec.~5]{DuTeUl19} for some historical remarks, that for $s>1/2$ the bound
\begin{equation}\label{widths}
	c_{s,d}n^{-s}(\log n)^{(d-1)s} \leq g_n(\Id_{s,d})\leq
	C_{s,d}n^{-s}(\log n)^{(d-1)(s+1/2)}\,
\end{equation}
holds asymptotically in $n\in \N$\,.
Note, that there is a $d$-depending gap in the logarithm between upper and lower bound.

Recently, Krieg and M. Ullrich \cite{KrUl19} improved this bound by using a probabilistic technique to show that for $s>1/2$
$$
   g_n(\Id_{s,d}) \lesssim_{s,d} n^{-s}(\log n)^{(d-1)s+s}\,.
$$
Clearly, if $s<(d-1)/2$ then the gap in \eqref{widths} is reduced to $(\log n)^s$, which is still growing in $s$. In particular, there is no improvement if $d=2$. However, this result can be considered as a major progress for the research on the complexity of this problem. They disproved Conjecture 5.6.2. in \cite{DuTeUl19} for $p=2$ and $1/2 < s <(d-1)/2$. Indeed, the celebrated sparse grid points are now beaten by random points in a certain range for $s$. This again reflects the ``power of random information'', see \cite{HiKrNoPrUl19}.

Still it is worth mentioning that the sparse grids represent the best known deterministic construction what concerns the asymptotic order. Indeed, the guarantees are deterministic and only slightly worse compared to random nodes in the asymptotic regime. However, regarding preasymptotics the random constructions provide substantial advantages. The problem is somehow related to the recent efforts in compressed sensing. There the optimal RIP matrices are given as realizations of random matrices. Known deterministic constructions are far from being optimal.

In the present paper we prove that the sparse grids are beaten for the full range of $s>1/2$ whenever $d\geq 3$. In case $d=2$ our approach and the sparse grids have the same performance. Clearly, inserting \eqref{sigma_n} into the bound in Theorem \ref{thm9} gives
\begin{equation}\label{new}
   g_n(\Id_{s,d}) \lesssim_{s,d} n^{-s}(\log n)^{(d-1)s+1/2},
\end{equation}
which shortens the gap between upper and lower bound to $\sqrt{\log n}$. The best known lower bound is the one from \eqref{sigma_n}. It is neither clear whether the bound in \eqref{new} is sharp nor if it can be improved. So this framework might serve as a candidate for Remark \ref{rem63}. Therefore, the outstanding open question remains (see, e.g., \cite[Chapt.\ 5]{DuTeUl19} and the references therein) whether there is an intrinsic additional difficulty when restricting to algorithms based on function samples rather than Fourier coefficients. From a practical point of view sampling algorithms are  highly relevant since we usually have given discrete samples of functions. The question remains: are the asymptotic characteristics $\sigma_n$ and $g_n$ of the same order or do they rather behave like $\sigma_n = o(g_n)$? This represents a  fundamental open problem in hyperbolic cross approximation, see \cite[Outstanding Open Problem 1.4]{DuTeUl19}.

\paragraph{Acknowledgment.} The authors would like to thank V.N. Temlyakov for giving a series of talks at the Chemnitz Summer School on Applied Analysis where he brought the paper \cite{NiOlUl16} to their attention. Theorem \ref{thm_weaver} is a generalization of the main result in \cite{NiOlUl16}. After this preprint was finished, V.N. Temlyakov pointed out to the authors that the proof of Lemma 2.2 in the recent paper \cite{LiTe20} also yields a version of Theorem \ref{thm_weaver} above with different constants. The authors would further like to thank Mario Ullrich for a useful comment regarding the case distinction for computing the explicit constants in Theorem 2.3. Last but not least they thank David Krieg for useful remarks on Section 4. T.U.\ would like to acknowledge support by the DFG Ul-403/2-1.

\nocite{NiOlUl16,MaSpSr15,Na20,Tem93}

\end{document}